\documentclass[runningheads]{llncs}
\usepackage{graphicx}
\usepackage{amsmath,amssymb,url,color,amsfonts}
\usepackage{authblk,abstract}
\usepackage{comment}
\usepackage{tikz}
\usepackage{stmaryrd}
\usepackage{enumitem}
\usepackage{marvosym}
\usepackage{multicol}
\usepackage{turnstile}

\makeatletter
\providecommand*{\Dashv}{%
  \mathrel{%
    \mathpalette\@Dashv\vDash
  }%
}
\newcommand*{\@Dashv}[2]{%
  \reflectbox{$\m@th#1#2$}%
}
\makeatother

\newcommand{\meet}{\wedge}
\newcommand{\join}{\vee}
\newcommand{\var}{{\sf Var}}
\newcommand{\luc}{\text{\textbf{\L}}}
\newcommand{\eq}{\approx}
\newcommand{\m}{\bf}
\newcommand{\mvi}{MV(I)}
\newcommand{\sfi}{S4MV(I)}
\newcommand{\sfmv}{S4MV}
\newcommand{\tmv}{S4$_t$MV}
\newcommand{\varmvi}{${\sf MV}(I)$}
\newcommand{\varsfi}{${\sf S4MV}(I)$}
\newcommand{\varsfmv}{${\sf S4MV}$}
\newcommand{\vartmv}{${\sf S4_t MV}$}
\newcommand{\MV}{{\sf MV}}
\newcommand{\iffprod}{\ast}
\newcommand{\up}{\mathord{\uparrow}}

\newcommand{\sm}{\mathcal{L}(\Box)}
\newcommand{\tm}{\mathcal{L}(G,H)}

\newcommand{\bic}{\leftrightarrow}

\begin{document}
\title{Some modal and temporal translations of generalized basic logic\thanks{This project received funding from the European Research Council (ERC) under the European Union’s Horizon 2020 research and innovation program (grant agreement No. 670624).}}
%
%
\author{Wesley Fussner\inst{1} \and
William Zuluaga Botero\inst{1,2}}
\authorrunning{W. Fussner and W. Zuluaga Botero}
%
\institute{Laboratoire J.A. Dieudonn\'e, CNRS, and Universit\'e C\^ote d'Azur, France\\
\and Departamento de Matem\'atica, Facultad de Ciencias Exactas, Universidad Nacional del Centro, Argentina\\
\email{wfussner@unice.fr}\\
\email{wizubo@gmail.com}}
\maketitle    
\begin{abstract}
We introduce a family of modal expansions of \L{}ukasiewicz logic that are designed to accommodate modal translations of generalized basic logic (as formulated with exchange, weakening, and falsum). We further exhibit algebraic semantics for each logic in this family, in particular showing that all of them are algebraizable in the sense of Blok and Pigozzi. Using this algebraization result and an analysis of congruences in the pertinent varieties, we establish that each of the introduced modal \L{}ukasiewicz logics has a local deduction-detachment theorem. By applying Jipsen and Montagna's poset product construction, we give two translations of generalized basic logic with exchange, weakening, and falsum in the style of the celebrated G\"odel-McKinsey-Tarski translation. The first of these interprets generalized basic logic in a modal \L{}ukasiewicz logic in the spirit of the classical modal logic {\m S4}, whereas the second interprets generalized basic logic in a temporal variant of the latter.

\end{abstract}
\section{Introduction}

Generalized basic logic (see, e.g., \cite{BM2009,JM2006}) is a common fragment of intuitionistic propositional logic and H\'ajek's basic fuzzy logic \cite{H1998}. It originates from algebraic studies of substructural logics, and its algebraic models (viz. \emph{GBL-algebras}) provide a natural common generalization of lattice-ordered groups, Heyting algebras, and continuous t-norm based logic algebras (see \cite{GJ2009} for a survey). In this capacity, generalized basic logic has been highly influential in the development of residuated lattices \cite{GJKO2007}, which provide the algebraic semantics of substructural logics. Generalized basic logic has also been proposed as a model of flexible resources \cite{BM2009}, in keeping with resource-driven interpretations of substructural logics generally (see, e.g., \cite{PO1999}).

When extended by exchange, weakening, and falsum (as we do throughout the sequel), generalized basic logic may be regarded as an `intuitionistic' variant of H\'ajek's basic logic. In this formulation, generalized basic logic is related to \L{}ukasiewicz logic \cite{CDM2000} in much the same way that intuitionistic logic is related to classical logic. For instance, generalized basic logic admits a Kripke-style relational semantics \cite{F2021} in which worlds are valued in MV-algebra chains, mirroring the well known Kripke semantics for intuitionistic logic (in which worlds are valued in the $2$-element Boolean algebra). It is evident from \cite{F2021} that generalized basic logic may be viewed as a fragment of a modal \L{}ukasiewicz logic, but the details of this modal connection are therein left implicit. On the other hand, \cite{F2021} generalizes the temporal flow semantics for basic logic \cite{ABM2009}, which is deployed in \cite{AGM2008} to obtain a modal translation of G\"odel-Dummett logic into an extension of Prior's tense logic \cite{P1957}. Inspired by this work, the present study makes the modal connection from \cite{F2021} explicit and offers modal and temporal translations of generalized basic logic into certain expanded \L{}ukasiewicz logics. Like \cite{AGM2008}, these translations are  directly analogous to the well known G\"odel-McKinsey-Tarski translation of intuitionistic logic into the classical modal logic {\m S4}, and are connected to the broader theory of modal companions of superintuitionistic logics. In addition to clarifying the role of modality in generalized basic logic, we expect that these results open up the application of tools from fuzzy modal logic (such as filtration \cite{EGR2017}) to the analysis of generalized basic logic and its extensions.

Our contributions are as follows. First, we introduce in Section~\ref{sec:logics} a family of modal \L{}ukasiewicz logics that serve as targets for our translations. This family includes both monounary modal systems, analogous to classical {\m S4}, as well as multimodal systems of temporal \L{}ukasiewicz logic. This investigation is rooted in algebraic logic, and in Section~\ref{sec:algebras} we provide pertinent information on algebras related to this study. In Section~\ref{sec:algebraization}, we demonstrate that all of the logics introduced in Section~\ref{sec:logics} are algebraizable in the sense of Blok and Pigozzi (see \cite{BP1989}) and that the algebras introduced in Section~\ref{sec:residuated lattices} provide their equivalent algebraic semantics. Equipped with this algebra-logic bridge, Section~\ref{sec:deduction} puts our algebraization theorem to work and establishes a local deduction detachment-theorem for our modal \L{}ukasiewicz logics. The work of Section~\ref{sec:deduction} is based on an analysis of congruences in the varieties of algebras introduced in Section~\ref{sec:residuated lattices}, and in particular establishes the congruence extension property for each of these varieties. Finally, in Section~\ref{sec:translations} we introduce two translations of generalized basic logic, one into a \L{}ukasiewicz version of {\m S4} and the other into a temporal \L{}ukasiewicz logic. These translations both rely on the poset product construction of Jipsen and Montagna (see, e.g., \cite{JM2010}).

\section{Generalized basic logic and fuzzy modal logics}\label{sec:logics}
This section introduces the logical systems of our inquiry. The logics discussed in this paper are all defined over supersets of the propositional language $\mathcal{L}$ consisting of the binary connectives $\meet,\join,\cdot,\to$ and the constants $0,1$. To the basic language $\mathcal{L}$ we will adjoin a set of box-like unary modal connectives. More specifically, given a set $I$ of unary connective symbols with $I\cap\mathcal{L}=\emptyset$, we define a language $\mathcal{L}(I) = \mathcal{L}\cup I$. We further fix a countably-infinite set $\var$ of propositional variables, and define the set $Fm_{\mathcal{L}(I)}$ of $\mathcal{L}(I)$-formulas over $\var$ in the usual way. An \emph{${\mathcal{L}}(I)$-equation} is an ordered pair in $(\varphi,\psi)\in Fm_{\mathcal{L}(I)}$, and we usually denote the equation $(\varphi,\psi)$ by $\varphi\eq\psi$. The set of all ${\mathcal{L}}(I)$-equations is denoted by $Eq_{{\mathcal{L}}(I)}$. All of the logics we consider may be defined by Hilbert-style calculi using various selections from the axiom schemes and deduction rules depicted in Figure~\ref{fig:calculus}. Observe that in Figure~\ref{fig:calculus} each of (K$_\Box$), (P$_\Box$), (M$_\Box$), ($1_\Box $), ($0_\Box $), (T$_\Box$), (\MVFour$_\Box$), (GP), (HF), and ($\Box$-Nec) gives a family of axiom schemes/rules parameterized by the unary connectives $\Box$, $G$, $H$. Note that we write $\varphi\bic\psi$ for $(\varphi\to\psi)\meet (\psi\to\varphi)$ and $\neg\varphi$ for $\varphi\to 0$ as usual.

\begin{figure}
\setlength{\columnsep}{2.5cm}
\begin{multicols}{2}
    \begin{itemize}[align=left]
	\item[\hspace{1.3cm}] {\bf Axiom schemes}
	\item[(A1)] $\varphi \to \varphi$
	\item[(A2)] \mbox{$(\varphi \to \psi) \to ((\psi \to \chi) \to (\varphi \to \chi))$}
	\item[(A3)] $(\varphi \cdot \psi) \to (\psi \cdot \varphi)$
	\item[(A4)] $(\varphi \cdot \psi) \to \psi$
	\item[(A5)] \mbox{$(\varphi \to (\psi \to \chi)) \to ((\varphi \cdot\psi) \to \chi))$}
	\item[(A6)] \mbox{$((\varphi \cdot\psi) \to \chi)) \to (\varphi \to (\psi \to \chi))$}
	\item[(A7)] $(\varphi \cdot (\varphi \to \psi)) \to (\varphi \wedge \psi)$
	\item[(A8)] $(\varphi \wedge \psi) \to (\varphi \cdot (\varphi \to \psi))$
	\item[(A9)] $(\varphi \wedge \psi) \to (\psi \wedge \varphi)$
	\item[(A10)] $\varphi \to (\varphi \vee \psi)$
	\item[(A11)] $\psi \to (\varphi \vee \psi)$
	\item[(A12)] \mbox{$((\varphi \to \psi) \wedge (\chi \to \psi)) \to ((\varphi \vee \chi) \to \psi)$}
	\item[(A13)] $0 \to \varphi$
	\item[(A14)] $(\varphi\to\psi)\join (\psi\to\varphi)$
	\item[(A15)] $\neg\neg \varphi \leftrightarrow \varphi$
	\item[(K$_\Box$)] \mbox{$\Box (\varphi \to \psi)\to (\Box \varphi\to\Box \psi)$}
	\item[(P$_\Box$)] $\Box (\varphi\cdot \psi) \leftrightarrow \Box \varphi \cdot \Box \psi$
	\item[(M$_\Box$)] $\Box (\varphi \meet \psi) \leftrightarrow \Box \varphi \meet \Box \psi$
	\item[($1_\Box $)] $\Box 1 \leftrightarrow 1$
	\item[($0_\Box $)] $\Box 0 \leftrightarrow 0$
	\item[(T$_\Box$)] $\Box \varphi \to \varphi$
	\item[(\MVFour$_\Box$)] $\Box \varphi \to \Box\Box \varphi$
	\item[(GP)] $\varphi\to G\neg H\neg\varphi$
	\item[(HF)] $\varphi\to H\neg G\neg\varphi$
	\item[\vspace{\fill}]
	\item[\hspace{1.3cm}] {\bf Rules}
	\item[(MP)] $\varphi, \varphi \to \psi \vdash \psi$
	\item[($\Box$-Nec)] $\varphi\vdash\Box\varphi$
\end{itemize}
\end{multicols}

\caption{Axiom schemes and rules for the logics considered.}
\label{fig:calculus}
\end{figure}

From \cite{BM2009}, \emph{generalized basic logic with exchange, weakening, and falsum} is the logic defined over $\mathcal{L}$ by the calculus with (A1)--(A13) and the modus ponens rule (MP). We denote this logic by $\bf GBL$. Additionally including the prelinearity axiom (A14) yields H\'ajek's basic fuzzy logic \cite{H1998}, which we denote by $\bf BL$. It follows from \cite{CT2003} that including both (A14) and (A15) gives an axiomatization of the infinite-valued \L{}ukasiewicz logic $\luc$ (see, e.g., \cite{CDM2000}).

We will consider a number of different modal expansions of $\luc$ in this study. For an arbitrary set $I$ of unary connective symbols disjoint from $\mathcal{L}$, we denote by $\luc(I)$ the logic with language $\mathcal{L}(I)$, axiom schemes (A1)--(A15), (K$_\Box$), (P$_\Box$), (M$_\Box$), ($1_\Box$), and ($0_\Box$) (where $\Box$ ranges over $I$ in all of the preceding axiom schemes), and rules (MP) and ($\Box$-Nec) (where again $\Box$ ranges over $I$). We denote by ${\bf S4\luc}(I)$ the logic resulting from adding to $\luc(I)$ the axiom schemes (T$_\Box$) and (\MVFour$_\Box$) for all $\Box\in I$. If $I=\{\Box\}$ is a singleton, we write ${\bf S4\luc}$ for ${\bf S4\luc}(I)$. If $I=\{G,H\}$, then the logic defined by adding to ${\bf S4\luc}(I)$ the axioms (GP) and (HF) will be denoted by ${\bf S4_t\luc}$.

The logic ${\bf S4\luc}$ is a fuzzy analogue of the classical modal logic $\bf S4$, whereas ${\bf S4_t\luc}$ is a temporal variant of ${\bf S4\luc}$ inspired by Prior's tense logic \cite{P1957}. The names of the axioms (GP) and (HF) derive from the fact that---as usual in tense logic---we define modal diamond connectives $P$ and $F$ as abbreviations for $\neg H\neg$ and $\neg G\neg$, respectively. The typical intended interpretations of the modals $G,P,H,F$ are:
\begin{itemize}
\item $G\varphi$: ``It is always \textbf{g}oing to be the case that $\varphi$.''
\item $P\varphi$: ``It was true at one point in the \textbf{p}ast that $\varphi$.''
\item $H\varphi$: ``It always \textbf{h}as been the case that $\varphi$.''
\item $F\varphi$: ``It will be true at some point in the \textbf{f}uture that $\varphi$.''
\end{itemize}
In Section~\ref{sec:translations}, we will exhibit translations of ${\bf GBL}$ into each of ${\bf S4\luc}$ and ${\bf S4_t\luc}$. These translations closely mirror the G\"odel-McKinsey-Tarski translation of propositional intuitionistic logic into {\bf S4}. Intuitively, ${\bf S4\luc}$ is a modal companion of $\bf GBL$ (see \cite{CZ1992}). On the other hand, our translation into ${\bf S4_t\luc}$ generalizes the translation presented in \cite{AGM2008} of G\"odel-Dummett logic into Prior's classical tense logic.

Given a logic $\bf L$, we denote by $\vdash_{\bf L}$ the consequence relation corresponding to $\bf L$ (see \cite{F2016} for background on consequence relations). As one may anticipate from the presence of the axioms (K$_\Box$) and $\Box$-necessitation rules, the logics we have introduced above turn out to be algebraizable in the sense of Blok and Pigozzi \cite{BP1989} (see Theorem~\ref{thm:algebraization}).

\section{Algebraic semantics for $\luc (I)$ and its extensions}\label{sec:algebras}

We now turn to providing algebraic semantics for the logics introduced in Section~\ref{sec:logics}. In Section~\ref{sec:residuated lattices} we describe the pertinent algebraic structures, and then in Section~\ref{sec:algebraization} we give the algebraization results for the logics we have introduced. We assume familiarity with the basics of universal algebra \cite{BS1981}, residuated lattices \cite{GJKO2007}, and abstract algebraic logic \cite{F2016}, but where possible we provide specific references to some key background results that we invoke without full discussion.

\subsection{Residuated lattices and their expansions}\label{sec:residuated lattices}

An algebra $(A,\meet,\join,\cdot,\to,0,1)$ is called a \emph{bounded commutative integral residuated lattice} if $(A,\meet,\join,0,1)$ is a bounded lattice, $(A,\cdot,1)$ is a commutative monoid, and for all $x,y,z\in A$,
$$x\cdot y\leq z \iff x\leq y\to z.$$
We usually abbreviate $x\cdot y$ by $xy$.

By a \emph{GBL-algebra} we mean a bounded integral commutative residuated lattice that satisfies the \emph{divisibility} identity $x(x\to y) \eq x\meet y$.\footnote{Most studies refer to these algebras as \emph{bounded commutative GBL-algebras} or \emph{GBL$_{ewf}$-algebras}. Because we always assume boundedness and commutativity, we call them GBL-algebras in order to simplify terminology.} A \emph{BL-algebra} is a GBL-algebra that satisfies $(x\to y)\join (y\to x)\eq 1$, and an \emph{MV-algebra} is a BL-algebra that satisfies $\neg\neg x \eq x$. The following definition gives the various classes of MV-algebra expansions that algebraize the logics of Section~\ref{sec:logics}.
\begin{definition}\label{def:MV(I)-alg}
Let $I$ be a set of unary function symbols. We say that an algebra ${\m A} = (A,\meet,\join,\cdot,\to,0,1, \{\Box\}_{\Box\in I})$ is an \emph{\mvi-algebra} provided that:
\begin{enumerate}\label{def:MV-alg}
\item $(A,\meet,\join,\cdot,\to,0,1)$ is an MV-algebra.
\item For every $\Box\in I$, $\Box$ is a $\{\wedge,\cdot,0,1\}$-endomorphism of $(A,\meet,\join,\cdot,\to,0,1)$.
\end{enumerate}
If additionally $\Box$ is an interior operator for every $\Box\in I$, then we say that $\mathbf{A}$ is an \emph{\sfi-algebra}. An \emph{\sfmv-algebra} is an \sfi-algebra where $I=\{\Box\}$ is a singleton. An \sfi-algebra for $I=\{G,H\}$ is called an \emph{\tmv-algebra} if the map $P$ defined by $P(x)= \neg H(\neg x)$ is the lower residual of $G$, i.e., for every $x,y\in A$,
$$x\leq G(y) \iff P(x) \leq y.$$
In each \tmv-algebra, we also abbreviate $\neg G(\neg x)$ by $F(x)$.
\end{definition}

The following summarizes some technical facts regarding \tmv-algebras. Its proof is straightforward and we omit it.

\begin{lemma}\label{lem:identities in tmv}
Let $\mathbf{A}$ be an \tmv-algebra and let $x,y\in A$. Then:
\begin{enumerate}
\item $P(x\join y) = P(x)\join P(y)$.
\item $P(0)=0$ and $P(1)=1$.
\item $x\leq H(y)$ if and only if $F(x)\leq y$.
\item $F(x\join y) = F(x)\join F(y)$.
\item $F(1)=1$ and $F(0)=0$.
\item $x\to GP(x) = 1$ and $PG(x)\to x=1$.
\item $x\to HF(x) = 1$ and $FH(x)\to x=1$. 
\item $P$ and $F$ are closure operators.
\end{enumerate}
\end{lemma}

It is well known that bounded commutative integral residuated lattices form a variety, and hence so do the classes of GBL-algebras, BL-algebras, and MV-algebras. We denote these varieties by $\sf GBL$, $\sf BL$, and $\sf MV$, respectively.

\begin{lemma}\label{lem: tMV is a variety}
Let $I$ be a set of unary function symbols with $\mathcal{L}\cap I=\emptyset$. The class of \mvi-algebras forms a variety, and the class of \sfi-algebras is a subvariety of the latter. Moreover, the class of \tmv-algebras forms a subvariety of the variety of S4MV(G,H)-algebras.
\end{lemma}

\begin{proof}
Clearly, the stipulation that each $\Box\in I$ is a $\{\meet,\cdot,0,1\}$-endomorphism is an equational property. Since $\sf MV$ is a variety, it follows that the class of \mvi-algebras forms a variety as well.  The stipulation that $\Box\in I$ is an interior operator is axiomatized relative to the defining conditions of \mvi-algebras by the identities $\Box\Box x \eq \Box x$ and $x\meet \Box x \eq \Box x$ since the monotonicity of $\Box$ follows from its being a $\meet$-endomorphism. Thus the class of \sfi-algebras is a subvariety of the variety of \mvi-algebras.

To see that the class of \tmv-algebras forms a subvariety of the variety of $S4MV(G,H)$-algebras, it is enough to prove that an S4MV(G,H)-algebra is a \tmv-algebra if and only if it satisfies the identities $x\to GP(x)\eq 1$ and $x\to HF(x)\eq 1$. Each \tmv-algebra is an S4MV(G,H)-algebra satisfying these identities by Lemma~\ref{lem:identities in tmv}(6,7). Conversely, an S4MV(G,H)-algebra satisfying these identities also satisfies $x\leq GP(x)$ and $x\leq HF(x)$ by residuation. Because $\neg$ is an order-reversing involution, $x\leq HF(x)$ is equivalent to $PG(x)\leq x$. Since $P$ and $G$ are monotone maps satisfying $x\leq GP(x)$ and $PG(x)\leq x$, it follows that $x\leq G(y)$ if and only if $P(x)\leq y$ (see, e.g., \cite[Lemma 3.2]{GJKO2007}). The result follows.
\end{proof}

We denote the varieties of \mvi-algebras, \sfi-algebras, \sfmv-algebras, and \tmv-algebras by ${\sf MV}(I)$, ${\sf S4MV}(I)$, ${\sf S4MV}$, and ${\sf S4_tMV}$, respectively.

\subsection{Algebraization}\label{sec:algebraization}

We now discuss algebraization of the logics of Section~\ref{sec:logics}. Each of the logics ${\bf GBL}$, ${\bf BL}$, and $\luc$ is algebraizable with the sole defining equation $\varphi\eq 1$ and sole equivalence formula $\varphi\leftrightarrow\psi$ (see, e.g., \cite{GJKO2007}). The equivalent variety semantics for $\bf GBL$, $\bf BL$, and $\luc$ are, respectively, the varieties $\sf GBL$, $\sf BL$, and $\sf MV$. The following lemma is a key ingredient in obtaining the algebraizability of the logics of Section~\ref{sec:logics}.

\begin{lemma}\label{lem: Expansions Lukasiewicz}
Let $I$ be a set of unary connectives with $I\cap\mathcal{L}=\emptyset$, and let $\bf L$ be an extension of $\emph{\luc}(I)$. Then $\varphi\leftrightarrow \psi \vdash_{\bf L} \Box \varphi \leftrightarrow \Box \psi$ for each $\Box\in I$.
\end{lemma}

\begin{proof}
Let $\Box \in I$. Note that $\varphi\leftrightarrow\psi\vdash_{\luc} \varphi\to\psi$ and $\varphi\leftrightarrow\psi\vdash_{\luc} \psi\to\varphi$, so $\varphi\leftrightarrow\psi\vdash_{\bf L} \varphi\to\psi,\psi\to\varphi$ as well. Applying ($\Box$-Nec) gives $\varphi\leftrightarrow\psi\vdash_{\bf L} \Box(\varphi\to\psi),\Box(\psi\to\varphi)$, so using (K$_\Box$) and (MP) gives $\varphi\leftrightarrow\psi\vdash_{\bf L} \Box\varphi\to\Box\psi,\Box\psi\to\Box\varphi$. Now $\varphi,\psi\vdash_{\luc}\varphi\meet\psi$ gives us that $\varphi,\psi\vdash_{\bf L}\varphi\meet\psi$, so it follows that $\varphi\leftrightarrow\psi\vdash_{\bf L}\Box\varphi\leftrightarrow\Box\psi$ as desired.
\end{proof}

The following gives our main result on algebraization.

\begin{theorem}\label{thm:algebraization}
Let $I$ be a set of unary connectives with $\mathcal{L}\cap I =\emptyset$. Then:
\begin{enumerate}
\item $\emph{\luc}(I)$ is algebraizable with the sole defining equation $\varphi\eq 1$ and sole equivalence formula $\varphi\leftrightarrow\psi$, and consequently so are ${\bf S4\emph{\luc}}(I)$, $\bf S4\emph{\luc}$, and $\bf S4_t\emph{\luc}$.
\item The equivalent variety semantics for $\emph{\luc}(I)$, ${\bf S4\emph{\luc}}(I)$, $\bf S4\emph{\luc}$, and $\bf S4_t\emph{\luc}$ are, respectively, \varmvi, \varsfi, \varsfmv, and \vartmv.
\end{enumerate}
\end{theorem}

\begin{proof}
1. It follows from \cite[Theorem 4.7]{BP1989} that a logic $\m L$ expanding $\luc$ by a set of connectives $\Omega$ is algebraizable if for every $n$-ary $\omega\in\Omega$ we have
$$\varphi_0\leftrightarrow\psi_0,\dots,\varphi_{n-1}\leftrightarrow\psi_{n-1}\vdash_{\m L} \omega(\varphi_0,\dots,\varphi_{n-1})\leftrightarrow \omega(\psi_0,\dots,\psi_{n-1}).$$
Moreover, in this case ${\m L}$ is algebraizable with sole defining equation $\varphi\eq 1$ and sole equivalence formula $\varphi\leftrightarrow\psi$. The result for ${\luc}(I)$ is thus immediate from Lemma~\ref{lem: Expansions Lukasiewicz}. The claim for ${\bf S4{\luc}}(I)$, $\bf S4{\luc}$, and $\bf S4_t{\luc}$ follows promptly because each of the latter logics is an axiomatic extension of ${\luc}(I)$ for some $I$.

2. By \cite[Theorem 2.17]{BP1989}, the quasivariety $\sf K$ algebraizing $\luc(I)$ is axiomatized by the following quasiequations: $\varphi\approx 1$ for all instances $\varphi$ of the axiom schemes given in the calculus for $\luc(I)$; $x\leftrightarrow x \approx 1$; $\varphi, \varphi\to\psi$ implies $\psi$; $\varphi$ implies $\Box\varphi$; and $x\leftrightarrow y \approx 1$ implies $x\approx y$. It is easy to see from Definition~\ref{def:MV-alg} and the fact that $\sf MV$ algebraizes $\luc$ that all of these quasiequations are valid in \varmvi. Thus ${\sf MV}(I)\subseteq {\sf K}$. For the reverse inclusion, it suffices to show that all the defining equations of $\MV(I)$ follow from this list of quasiequations. Let ${\m A}\in\sf K$. That the $\{\meet,\join,\cdot,\to,0,1\}$-reduct of $\m A$ is an MV-algebra is immediate from the fact that $\MV$ algebraizes $\luc$. On the other hand, for each $\Box\in I$ the equations $\Box (x\cdot y)\leftrightarrow \Box x\cdot\Box y \eq 1$, $\Box (x\meet y)\leftrightarrow \Box x\cdot\Box y \eq 1$, $\Box 1\leftrightarrow 1 \eq 1$, and $\Box 0\leftrightarrow 0 \eq 0$ appear in the list of quasiequations, and together these imply that $\Box$ is a $\{\meet,\cdot,0,1\}$-homomorphism of $\m A$ for each $\Box\in I$. Thus $\sf K\subseteq \MV(I)$, giving equality. The result for the axiomatic extensions ${\bf S4{\luc}}(I)$, $\bf S4{\luc}$, and $\bf S4_t{\luc}$ follows by applying the formula-to-equation translation $\varphi\mapsto\varphi\eq 1$ to each formula $\varphi$ axiomatizing the given logic relative to $\luc(I)$.
\end{proof}

Recall that if $\sf K$ is a class of similar algebras and $\Theta\cup\{\epsilon\eq\delta\}$ is a set of equations in the type of $\sf K$, then $\Theta\models_{\sf K}\epsilon\eq\delta$ means that for every ${\m A}\in{\sf K}$ and every assignment $h$ of variables into $\m A$, if $h(\alpha)=h(\beta)$ for every $\alpha\eq\beta\in\Theta$, then $h(\epsilon)=h(\delta)$. Thanks to the finitarity of $\luc(I)$, the following is a direct consequence of Theorem~\ref{thm:algebraization} (see \cite[Corollary 3.40]{F2016}). 

\begin{corollary}\label{cor:alg application}
Let $I$ be a set of unary connectives with $\mathcal{L}\cap I=\emptyset$. There is a dual lattice isomorphism between the lattice of finitary extensions of $\emph{\luc}(I)$ and the lattice of subquasivarieties of $\MV(I)$, which restricts to a dual lattice isomorphism between the lattice of axiomatic extensions of $\emph{\luc}(I)$ and the lattice of subvarieties of $\MV(I)$. Moreover, suppose that $\m L$ is a finitary extension of $\emph{\luc}(I)$, and let $\sf K$ be the equivalent algebraic semantics of $\m L$. Then for any set $\Gamma\cup\{\varphi\}\subseteq Fm_{\mathcal{L}(I)}$ and any set $\Theta\cup\{\epsilon\eq\delta\}\subseteq Eq_{\mathcal{L}(I)}$:
\begin{enumerate}
\item $\Gamma\vdash_{\m L}\varphi \iff \{\gamma\eq 1 : \gamma\in\Gamma\}\models_{\sf K} \varphi\eq 1$.
\item $\Theta\models_{\sf K} \epsilon\eq\delta \iff \{\alpha\leftrightarrow\beta : \alpha\eq\beta\in\Theta\}\vdash_{\m L} \epsilon\leftrightarrow\delta$.
\end{enumerate}
In particular, this holds if ${\m L}\in\{{\emph{\luc}(I)},{\bf S4}\emph{\luc}(I), {\bf S4}\emph{\luc},{\bf S4_t}\emph{\luc}\}$.
\end{corollary}

\section{Characterizing filters and a deduction theorem}\label{sec:deduction}
If ${\m L}$ is an algebraizable logic, there is a well known connection between the theories of ${\m L}$, the deductive filters of algebraic models of ${\m L}$, and the congruence relations of the equivalent algebraic semantics of ${\m L}$ (see, e.g., \cite{GO2006,F2016}). Armed with the algebraizability results of Section~\ref{sec:logics}, we now provide an analysis of congruences in the algebraic semantics given in Section~\ref{sec:algebras}. We also use this description to establish local-deduction detachment theorems for the modal \L{}ukasiewicz logics we have introduced. The following is key in our description of congruences.

\begin{definition}
Let $\mathbf{A}$ be an \mvi-algebra. We say that a non-empty subset $\frak{f}$ of $A$ is an \emph{I-filter} provided that $\frak{f}$ is an up-set, $\frak{f}$ is closed under $\cdot$, and $\frak{f}$ is closed under each $\Box\in I$.
\end{definition}

Let $\mathbf{A}\in \MV(I)$. We define a term operation $\iffprod$ by $x\iffprod y = (x\rightarrow y)(y\rightarrow x)$. We also write $\mathsf{Fi}(\mathbf{A})$ for the poset of I-filters of $\mathbf{A}$ ordered by inclusion and $\mathsf{Con}({\m A})$ for the congruence lattice of ${\m A}$. 

\begin{lemma}\label{lem:Congruence Filters}
Let $\mathbf{A}$ be an \mvi-algebra, $\frak{f}\in \mathsf{Fi}(\mathbf{A})$, and $\theta\in \mathsf{Con}(\mathbf{A})$. Then the following hold:
\begin{enumerate}
\item $\frak{f}_{\theta}=1/\theta$ is an I-filter of $\mathbf{A}$.
\item The set $\theta_{\frak{f}}=\{(x,y)\in A^{2}\colon x\iffprod y\in \frak{f} \} = \{(x,y)\in A^2\colon x\leftrightarrow y\in\frak{f}\}$ is a congruence on $\mathbf{A}$.
\item The maps $\frak{f}\mapsto \theta_{\frak{f}}$, $\theta \mapsto \frak{f}_{\theta}$ define mutually-inverse poset isomorphisms between $\mathsf{Con}(\mathbf{A})$ and $\mathsf{Fi}(\mathbf{A})$. Consequently, $\mathsf{Fi}(\mathbf{A})$ is a lattice and these poset isomorphisms are lattice isomorphisms.
\end{enumerate}
\end{lemma}

\begin{proof}
1. Note that $\frak{f}_{\theta}$ is a deductive filter of the MV-algebra reduct of ${\m A}$ (see, e.g., \cite[Section 3.6]{GJKO2007}), so it suffices to show that $\frak{f}_{\theta}$ is closed under $\Box$ for every $\Box\in I$. Observe that if $(1,x)\in\theta$ then since $\theta$ is a congruence we have $(\Box 1,\Box x )\in\theta$. But since $\Box 1=1$, it follows that $\Box x\in\frak{f}_{\theta}$ as desired.

2. Observe first that $x\iffprod y\in\frak{f}$ if and only if $x\leftrightarrow y\in\frak{f}$, so the two sets displayed are equal. Now since $\frak{f}$ is in particular a deductive filter of the MV-algebra reduct ${\m A}$, it is immediate that $\theta_\frak{f}$ respects all of the operations except for possibly those belonging to $I$. To show that $\theta_\frak{f}$ respects these as well, it suffices to show the result for every $\Box \in I$. Suppose that $(x,y)\in\theta_\frak{f}$, i.e., $x\iffprod y\in\frak{f}$. Since $\frak{f}$ is closed under $\Box$, and every $\Box \in I$ preserves $\cdot$, we have $\Box(x\to y)\cdot \Box(y\to x)\in\frak{f}$. Residuation and the fact that $\Box$ preserves $\cdot$ gives $\Box(x\to y)\leq \Box x \to \Box y $ and $\Box(y\to x)\leq \Box y \to \Box x $, so $$\Box(x\iffprod y)=\Box(x\to y) \Box(y\to x)\leq (\Box x \to \Box y)(\Box y \to \Box x ).$$ 
Since $\frak{f}$ is an up-set, we get $\Box x\iffprod \Box y \in\frak{f}$. Hence $(\Box x ,\Box y )\in\theta_\frak{f}$ as required.

3. Direct computation shows $\frak{f}=\frak{f}_{\theta_{\frak{f}}}$ and $\theta_{\frak{f}_{\theta}}=\theta$ for every I-filter $\frak{f}$ and congruence $\theta$. Moreover, the given maps are clearly monotone. It follows that $\mathsf{Con}(\mathbf{A})$ and $\mathsf{Fi}(\mathbf{A})$ are isomorphic as posets. Because $\mathsf{Fi}(\mathbf{A})$ is isomorphic to the lattice $\mathsf{Con}(\mathbf{A})$, we obtain that $\mathsf{Fi}(\mathbf{A})$ is a lattice that is isomorphic to $\mathsf{Con}(\mathbf{A})$.
\end{proof}

The following gives a description of congruence generation in \varmvi.

\begin{definition}
Let ${\m A}$ be an \mvi-algebra and let $X\subseteq A$.
\begin{enumerate}
\item An \emph{I-block} is a nonempty word in the alphabet $I$. We denote the set of I-blocks by $\mathcal{B}_{I}$.
\item $\mathsf{Fg}^{\mathbf{A}}(X)=\up\{M_1x_{1}\cdot...\cdot M_n x_{n} \colon x_{1},...,x_{n}\in X\text{ and }M_1,...,M_n\in\mathcal{B}_{I}\}$. 
\end{enumerate}
\end{definition}

\begin{lemma}\label{lem:Generated Filter MVI}
The set $\mathsf{Fg}^{\mathbf{A}}(X)$ is the least I-filter of $\mathbf{A}$ containing $X$.
\end{lemma}
\begin{proof}
It is clear that $\mathsf{Fg}^{\mathbf{A}}(X)$ is an up-set. Note that if $y,y'\in \mathsf{Fg}^{\mathbf{A}}(X)$ then there exist $M_1,...,M_n,M'_1,...,M'_k\in\mathcal{B}_{I}$ and $x_1,...,x_n,x'_1,...,x'_k\in X$ with $M_1x_1\cdot...\cdot M_nx_n\leq y$ and $M'_1 x'_1\cdot...\cdot M'_k x'_k\leq y'$, whence $M_1x_1\cdot...\cdot M_nx_n\cdot M'_1 x'_1\cdot...\cdot M'_k x'_k\leq y\cdot y'$ since $\cdot$ preserve the order in each coordinate. It follows that $y\cdot y'\in\mathsf{Fg}^{\mathbf{A}}(X)$. To see that $\mathsf{Fg}^{\mathbf{A}}(X)$ is closed under every $\Box\in I$, observe that if $M_1x_1\cdot...\cdot M_nx_n\leq y$ then by the isotonicity of $\Box$ we have $\Box M_1x_1\cdot...\cdot \Box M_nx_n\leq \Box y$. As each $\Box M_i$ is an I-block, it follows that $\Box y\in\mathsf{Fg}^{\mathbf{A}}(X)$. 

It remains to check that $\mathsf{Fg}^{\mathbf{A}}(X)$ is the least among the I-filters containing $X$. Suppose that $\frak{f}$ is an I-filter and that $X\subseteq\frak{f}$. If $y\in\mathsf{Fg}^{\mathbf{A}}(X)$, then there exist $M_1,...,M_n\in\mathcal{B}_{I}$ and $x_1,...,x_n\in X$ such that $M_1x_1\cdot...\cdot M_nx_n\leq y$. Note that $x_1,...,x_n\in\frak{f}$, and since $\frak{f}$ is closed under $\Box$ for every $\Box\in I$, we have that $Mx\in\frak{f}$ for every $M\in\mathcal{B}_{I}$ and every $x\in\frak{f}$. In particular, this implies that $M_1 x_1,...,M_n x_n\in\frak{f}$. Thus $y\in\frak{f}$ since $\up\frak{f}=\frak{f}$, so $\mathsf{Fg}^{\mathbf{A}}(X)\subseteq\frak{f}$ as claimed.
\end{proof}

We abbreviate $\mathsf{Fg}^{\mathbf{A}}(\{x_{1},...,x_{n}\})$ by $\mathsf{Fg}^{\mathbf{A}}(x_{1},...,x_{n})$. Also, for an algebra ${\m A}$ and $x,y\in A$, we denote by ${\mathsf{Cg}^{\mathbf{A}}(x,y)}$ the congruence of ${\m A}$ generated by $(x,y)$.

\begin{lemma}\label{cor:Generated filter intersection}
Let $\mathbf{A}\in \MV(I)$, let $x,y\in A$, let $Y\subseteq A$, and consider $X=\{(1,y)\colon a\in Y\}$. Then:
\begin{enumerate}
\item $\frak{f}_{\mathsf{Cg}^{\mathbf{A}}(x,y)}=\mathsf{Fg}^{\mathbf{A}}(x\iffprod y)=\mathsf{Fg}^{\mathbf{A}}(x\leftrightarrow y)$.
\item $\mathfrak{f}_{\mathsf{Cg}^{\mathbf{A}}(X)}=\mathsf{Fg}^{\mathbf{A}}(Y)$.
\end{enumerate}
\end{lemma}
\begin{proof}
1. Note that $\mathsf{Cg}^{\mathbf{A}}(x,y)=\bigcap \{\theta\in \mathsf{Con}(\mathbf{A})\colon (x,y)\in\theta\}$, and observe that for each $\theta\in\mathsf{Con}(\mathbf{A})$ we have $(x,y)\in\theta$ if and only if $x\iffprod y \in \frak{f}_{\theta}$. Hence from the isomorphism given by Lemma \ref{lem:Congruence Filters}(3) we obtain:
$$\frak{f}_{\mathsf{Cg}^{\mathbf{A}}(x,y)}=\bigcap \{\frak{f}\in \mathsf{Fi}(\mathbf{A})\colon x\iffprod y\in \frak{f}\}=\mathsf{Fg}^{\mathbf{A}}(x\iffprod y)=\mathsf{Fg}^{\mathbf{A}}(x\leftrightarrow y).$$
This proves 1.

2. Since $\mathsf{Cg}^{\mathbf{A}}(X)=\bigvee_{y\in Y}\mathsf{Cg}^{\mathbf{A}}(1,y)$, Lemma \ref{lem:Congruence Filters}(3) and item 1 imply
\[\mathfrak{f}_{\mathsf{Cg}^{\mathbf{A}}(X)}=\bigvee_{y\in Y} \mathfrak{f}_{\mathsf{Cg}^{\mathbf{A}}(1,y)}=\bigvee_{y\in Y} \mathsf{Fg}^{\mathbf{A}}(y)=\mathsf{Fg}^{\mathbf{A}}(\bigcup_{y\in Y} \{y\})=\mathsf{Fg}^{\mathbf{A}}(Y).\]
This proves 2.
\end{proof}

Recall that an algebra ${\m B}$ has the \emph{congruence extension property} (or \emph{CEP}) if for every subalgebra $\mathbf{A}$ of $\mathbf{B}$ and for any $\theta \in \mathsf{Con}(\mathbf{A})$, there exists $\xi \in \mathsf{Con}(\mathbf{B})$ such that $\xi\cap A^{2} = \theta$. A variety $\sf V$ is said to have the congruence extension property if each ${\m B}\in\sf V$ does.

\begin{theorem}\label{thm:tMV has CEP}
$\MV(I)$ has the congruence extension property.
\end{theorem}
\begin{proof}
Let $\mathbf{A}, \mathbf{B}$ be \mvi-algebras, and assume that $\mathbf{A}$ is a subalgebra of $\mathbf{B}$. From Lemma~\ref{lem:Congruence Filters}, it follows that proving the congruence extension property for $\MV(I)$ is equivalent to showing that every I-filter of $\mathbf{A}$ can be extended by an I-filter of $\mathbf{B}$. For this, let $\frak{f}\in \mathsf{Fi}(\mathbf{A})$ and set $\frak{g}=\mathsf{Fg}^{\mathbf{B}}(\frak{f})$. In order to prove $\frak{f}=\frak{g}\cap A$, let $y\in \frak{g}\cap A$. Then since $y\in\frak{g}$ there exist $M_{1},...,M_{n}\in \mathcal{B}_{I}$ and $x_{1},...,x_{n}\in \frak{f}$ such that $M_{1}(x_{1})\cdot ... \cdot M_{n}(x_{n})\leq y$. Since $\frak{f}$ is an I-filter of $\mathbf{A}$, we have $M_{j}(x_{j})\in \frak{f}$ for every $1\leq j\leq n$. As $y\in A$, it follows that $y\in \frak{f}$ and $\frak{g}\cap A\subseteq\frak{f}$. The reverse inclusion is obvious, and the result follows.
\end{proof}

Of course, the CEP persists in subvarieties of a variety with the CEP. Thus:
\begin{corollary}\label{cor:CEP subvarieties}
Each of \varsfi, \varsfmv, and {\vartmv} has the CEP.
\end{corollary}

The CEP has far-reaching logical consequences. Recall that a logic $\mathbf{L}$ has the \textit{local deduction-detachment theorem} (or \emph{LDDT}) if there exists a family $\{d_j(p,q)\colon j\in J\}$ of sets $d_j(p,q)$ of formulas in at most two variables such that for every set $\Gamma\cup \{\varphi,\psi\}$ of formulas in the language of $\m L$:
\begin{displaymath}
\begin{array}{ccc}
\Gamma, \varphi \vdash_{\mathbf{L}} \psi &  \Longleftrightarrow & \Gamma  \vdash_{\mathbf{L}} d_{j}(\varphi,\psi)\; \text{for some}\; j\in J.
\end{array}
\end{displaymath}
As a consequence of \cite[Corollary 5.3]{BP1991}, if $\m L$ is an algebraizable logic with equivalent variety semantics $\sf V$, then $\m L$ has the LDDT if and only if $\sf V$ has the CEP. Therefore from Theorem~\ref{thm:algebraization}, Theorem~\ref{thm:tMV has CEP}, and Corollary~\ref{cor:CEP subvarieties} we obtain:
\begin{corollary}\label{theo: S4tL satisfies LDDT}
Each of $\emph{\luc}(I)$, ${\bf S4{\emph\luc}}(I)$, $\bf S4{\emph\luc}$, and $\bf S4_t{\emph\luc}$ has the LDDT.
\end{corollary}

From our analysis of congruences in $\MV(I)$, we may give a more explicit rendering of this result. If $\mathsf{V}$ is a variety, we denote by $\mathbf{F}_{\mathsf{V}}(X)$ the $\sf V$-free algebra over $X$. Further, if $\varphi$ is a formula, denote by $\bar{\varphi}$ the image of $\varphi$ under the natural map ${\bf Fm}(X)\to \mathbf{F}_{\mathsf{V}}(X)$ from the term algebra ${\m Fm}(X)$ over $X$ onto $\mathbf{F}_{\mathsf{V}}(X)$. If $\Gamma$ is a set of formulas, also denote by $\bar{\Gamma}=\{\bar{\varphi} : \varphi\in\Gamma\}$. The following restates \cite[Lemma 2]{MMT2014}.
\begin{lemma}\label{lem: Technical lemma}
Let $\Theta\cup \{\varphi\eq\psi\}$ be a set of equations in the language of $\sf V$, and take $X$ to be the set of variables appearing in $\Theta\cup\{\varphi\eq\psi\}$. Then the following are equivalent:
\begin{enumerate}
\item $\Theta \models_{\mathsf{V}} \varphi \eq \psi$.
\item $(\bar{\varphi}, \bar{\psi})\in \bigvee_{\epsilon\eq\delta\in\Theta} \mathsf{Cg}^{\mathbf{F}_{\mathsf{V}}(X)}(\bar{\epsilon},\bar{\delta})$.
\end{enumerate}
\end{lemma}

\begin{theorem}\label{theo: S4tL satisfies parametrized LDDT}
Let $I$ be a set of unary connectives with $I\cap\mathcal{L}=\emptyset$, and suppose that ${\m L}$ is an axiomatic extension of $\emph{\luc}(I)$ that is algebraized by the subvariety $\sf V$ of $\MV(I)$. Further, let $\Gamma\cup \Delta \cup \{\psi\} \subseteq Fm_{\mathcal{L}(I)}$. Then $\Gamma, \Delta \vdash_{\m L} \psi$ if and only if for some $n\geq 0$ there exist I-blocks $M_{1},\ldots,M_{n}$ and $\psi_{1},\ldots,\psi_{n}\in \Delta$ such that $\Gamma \vdash_{\m L} \prod_{j=1}^{n}M_{j}(\psi_{j})\rightarrow \psi$.
\end{theorem}

\begin{proof}
We give the proof of the left-to-right direction; the proof of the converse is similar. From Corollary~\ref{cor:alg application}(1) and Lemmas~\ref{lem: Technical lemma} and \ref{lem:Congruence Filters} we obtain:
\begin{displaymath}
\begin{array}{ccll}
\Gamma, \Delta \vdash_{\m L} \psi &  \Longrightarrow & \{\alpha \approx 1 : \alpha \in \Gamma \cup \Delta\}\models_{\sf V} \psi \eq 1 &
\\
 &\Longrightarrow & (\bar{\psi}, 1)\in \bigvee_{\alpha \in \Gamma \cup \Delta} \mathsf{Cg}^{\mathbf{F}_{\sf V}(X)}(\bar{\alpha},1) &
\\
 &\Longrightarrow & \bar{\psi} \in \mathsf{Fg}^{{\mathbf{F}_{\sf V}}(X)}( \bar{\Gamma} \cup \bar{\Delta}), &
\end{array}
\end{displaymath}
where $X$ is the set of variables appearing in $\Gamma\cup \Delta \cup \{\psi\}$. From Lemma~\ref{lem:Generated Filter MVI} there exist $l\geq 0$, I-blocks $M_{1},\ldots,M_{l}$, and $\bar{\chi}_{1},\dots, \bar{\chi}_{l}\in \bar{\Gamma} \cup \bar{\Delta}$ such that $M_{1}(\bar{\chi}_{1})\cdot\ldots\cdot M_{l}(\bar{\chi}_{l})\leq \bar{\psi}$. Let $D=\{j\in\{1,\ldots,l\} : \bar{\chi}_j\in\bar{\Delta}\}$, and set $C=D\setminus\{1,\ldots,l\}$. Then by the commutativity of $\cdot$ we have
$$\prod_{j\in C}M_{j}(\bar{\chi}_{j})\cdot \prod_{k\in D}M_{k}(\bar{\chi}_{k}) = M_{1}(\bar{\chi}_{1})\cdot\ldots\cdot M_{l}(\bar{\chi}_{l})\leq \bar{\psi},$$
whence by residuation $\prod_{j\in C}M_{j}(\bar{\chi}_{j})\leq \prod_{k\in D}M_{k}(\bar{\chi}_{k})\to \bar{\psi}$. Applying Lemma \ref{lem:Generated Filter MVI} again gives $\prod_{k\in D}M_{k}(\bar{\chi}_{k})\to \bar{\psi} \in \mathsf{Fg}^{\mathbf{F}_{\sf V}(X)}(\bar{\Gamma})$. Hence by Lemmas~\ref{lem:Congruence Filters} and \ref{lem: Technical lemma} and Corollary~\ref{cor:alg application}(1) we obtain $\Gamma \vdash_{\m L} \prod_{k\in D}M_{k}({\chi}_{k})\to \psi$.
\end{proof}

Notice that the form of the local deduction-detachment theorem announced in Corollary~\ref{theo: S4tL satisfies LDDT} may be recovered from Theorem~\ref{theo: S4tL satisfies parametrized LDDT} by taking $\Delta=\{\varphi\}$ and taking $d_M(p,q) = Mp\to q$ for $M\in\mathcal{B}_I$.

In the monomodal logic $\m S4\luc$, I-blocks take an especially simple form. Because $I=\{\Box\}$ in this setting, each I-block $M$ is a finite, nonempty string of occurrences of $\Box$. Since $\Box$ is idempotent in \varsfmv, for each $\{\Box\}$-block $M$ we have that $Mx\eq\Box x$ is satisfied in \varsfmv. Due to this consideration and the fact that $\Box$ preserves $\cdot$, we may read off the following simplified form the LDDT for $\m S4\luc$:

\begin{corollary}
Let $\Gamma\cup \Delta \cup \{\psi\} \subseteq Fm_{\mathcal{L}(\Box)}$. Then $\Gamma, \Delta \vdash_{\m S4\emph{\luc}} \psi$ if and only if for some $n\geq 0$ there exist $\psi_{1},\ldots,\psi_{n}\in \Delta$ such that $\Gamma \vdash_{\m S4\emph{\luc}} \Box(\prod_{j=1}^{n}\psi_{j})\rightarrow \psi$.
\end{corollary}

If $I=\{\Box_{1},...\Box_{n}\}$ is finite, then particular forms of the LDDT can be achieved for ${\bf S4\luc}(I)$ and its extensions by defining an operator $\lambda(x)=\prod_{i=1}^{n} \Box_{i}(x)$. Powers of $\lambda$ are defined recursively by $\lambda^{0}(x)=x$ and $\lambda^{m+1}(x)=\lambda(\lambda^{m}(x))$ for $m>0$. I-filters of S4MV(I)-algebras may be characterized with powers of $\lambda$ instead of I-blocks. A full discussion of this alternative approach will appear in future work.

\section{Two translations}\label{sec:translations}

We now arrive at our main translation results. After discussing some necessary technical background regarding the Jipsen-Montagna poset product construction, we exhibit two translations. The first of these embeds $\m GBL$ into $\m S4\luc$, and is conceptually in the spirit of the classical G\"odel-McKinsey-Tarski translation of intuitionistic logic into $\m S4$. The second translation embeds $\sf GBL$ in $\m S4_t\luc$.

\subsection{Poset products} \label{sec:poset products}
The translation results of this paper rely heavily on the poset product construction of Jipsen and Montagna (see \cite{JM2009,JM2010}), which we now sketch. Our discussion of poset products is drawn mainly from \cite{F2021}, to which we refer the reader for a more detailed summary.

Let ${\m A}$ be a bounded commutative integral residuated lattice. A \emph{conucleus} on ${\m A}$ is an interior operator $\gamma$ on the lattice reduct of ${\m A}$ such that $\gamma(x)\gamma(y)\leq \gamma(xy)$ and $\gamma(x)\gamma(1)=\gamma(1)\gamma(x)=\gamma(x)$ for all $x,y\in A$. Given a conucleus $\gamma$ on ${\m A}$, the $\gamma$-image ${\m A}_\gamma = (A_\gamma,\meet_\gamma,\join,\cdot,\to_\gamma,0,\gamma(1))$ is a bounded commutative integral residuated lattice, where $A_\gamma=\gamma[A]$ and $x\star_\gamma y=\gamma(x\star y)$ for $\star\in\{\meet,\to\}$.

Now let $(X,\leq)$ be a poset, let $\{{\m A}_x : x\in X\}$ be an indexed collection of bounded commutative integral residuated lattices with a common least element $0$ and a common greatest element $1$, and let ${\m B} = \prod_{x\in X} {\m A}_x$. From \cite[Lemma 9.4]{JM2010}, one may define a conucleus on $\sigma$ on ${\m B}$ by 
\[ \sigma(f)(x) = \begin{cases}
      f(x) & \text{ if }f(y)=1\text{ for all }y>x \\
      0 & \text{ if there exists }y>x\text{ with } f(y)\neq 1.
   \end{cases}
\]
The algebra ${\m B}_\sigma$ is called the \emph{poset product} of $\{{\m A}_x : x\in X\}$, and is denoted $\prod_{(X,\leq)} {\m A}_x$. An element $f\in B_\sigma$ is called an \emph{antichain labeling} or \emph{ac-labeling}, and satisfies the condition that if $x,y\in X$ with $x<y$ then $f(x)=0$ or $f(y)=1$. The following is a direct consequence of \cite[Corollary 5.4(i)]{JM2010} and its proof.
\begin{lemma}\label{lem:JipsenMontagna}
Let ${\m A}$ be a GBL-algebra. Then there exists a poset $(X,\leq)$ and an indexed family $\{{\m A}_x : x\in X\}$ of totally ordered MV-algebras such that ${\m A}$ embeds in the poset product ${\m B}_\sigma$, where ${\m B}=\prod_{x\in X} {\m A}_x$.
\end{lemma}
Following \cite{J2009}, for a poset $(X,\leq)$ and indexed family $\{{\m A}_x : x \in X\}$ we introduce a map $\delta$ on ${\m B}=\prod_{x\in X} {\m A}_x$ by
\[ \delta(f)(x) = \begin{cases}
      f(x) & \text{ if }f(y)=0\text{ for all }y<x \\
      1 & \text{ if there exists }y<x\text{ with } f(y)\neq 0.
   \end{cases}
\]
The following lemma is crucial for our translation result.
\begin{lemma}\label{lem:poset products are tmv}
Let $(X,\leq)$ be poset, let $\{{\m A}_x : x\in X\}$ be an indexed family of bounded commutative integral residuated lattices, and set ${\m B}=\prod_{x\in X} {\m A}_x$ as above. Then:
\begin{enumerate}
\item $\sigma$ and $\neg\delta\neg$ are $\{\meet,\cdot,0,1\}$-endomorphisms of ${\m B}$.
\item For all $f,g\in B$, $f\leq\sigma(g)$ if and only if $\delta(f)\leq g$.
\item $\neg\delta\neg$ is an interior operator.
\item If additionally ${\m A}_x$ is an MV-algebra for all $x\in X$, then $({\m B},\sigma)$ is an \sfmv-algebra and $({\m B},\sigma,\neg\delta\neg)$ is \tmv-algebra.
\end{enumerate}
\end{lemma}

\begin{proof}
1. It is obvious that $\sigma(0)=0$ and $\sigma(1)=1$. Let $\star\in\{\meet,\cdot\}$, $x\in X$, and $f,g\in B$, and observe that if $y>x$ then $(f\star g)(y) = 1$ if and only if $f(y)=g(y)=1$. It follows that if $(f\star g)(y) = 1$ for all $y>x$, then $\sigma(f\star g)(x)=(f\star g)(x)=f(x)\star g(x)=\sigma(f)(x)\star\sigma(g)(x)$, and if otherwise then $\sigma(f\star g)(x) = 0 = \sigma(f)(x)\star\sigma(g)(x)$. Thus $\sigma(f\star g)=\sigma(f)\star\sigma(g)$. 

To prove that $\neg\delta\neg$ is a $\{\meet,\cdot,0,1\}$-endomorphism, again let $\star\in\{\meet,\cdot\}$, $x\in X$, and $f,g\in B$. Note that for all $y<x$ we have $\neg (f\star g)(y) = 0$ if and only if $(f\star g)(y)=1$, and as before this occurs if and only if $f(y)=g(y)=1$. Thus we have $\neg (f\star g)(y)=0$ for all $y<x$ if and only if $\neg f(y) = 0$ for all $y<x$ and $\neg g(y)=0$ for all $y<x$. Hence if $\neg (f\star g)(y)=0$ for all $y<x$, then we have $\neg\delta\neg (f\star g)(x) = \neg\neg (f\star g)(x)=f(x)\star g(x)=\neg\neg f(x)\star\neg\neg g(x)=\neg\delta\neg f(x)\star\neg\delta\neg g(x)$. On the other hand, if there exists $y<x$ with $\neg (f\star g)(y)\neq 0$, then $\neg\delta\neg (f\star g)(x)=\neg 1 = 0$, and $\neg\delta\neg f(x)\star\neg\delta\neg g(x) = 0$ since one of $\delta\neg f(y)$ or $\delta\neg g(y)$ must be $1$. Since $\neg\delta\neg 0 = 0$ and $\neg\delta\neg 1 = 1$ by direct calculation, item 1 follows.

2. Suppose $f\leq \sigma(g)$ and let $x\in X$. Since $\sigma$ is an interior operator, $f(x)\leq\sigma(g)(x)\leq g(x)$. If $\delta(f)(x)=f(x)$, then $\delta(f)(x)\leq g(x)$ is immediate. On the other hand, if $\delta(f)(x)\neq f(x)$ then there exists $y<x$ such that $f(y)\neq 0$. From $f\leq \sigma(g)$ we infer that $\sigma(g)(y)\neq 0$, so $\sigma(g)(x)=1$ since $\sigma(g)$ is an ac-labeling. Thus $\delta(f)(x)\leq 1 = \sigma(g)(x)=g(x)$. It follows that $\delta(f)\leq g$. The proof that $\delta(f)\leq g$ implies $f\leq\sigma(g)$ is similar.

3. It is easy to see that $\delta$ is a closure operator. From this and the fact that $\neg$ is an antitone involution, it is a straightforward calculation to show that $\neg\delta\neg$ is an interior operator.

4. Under the hypothesis, ${\m B}$ is a product of MV-algebras and is hence an MV-algebra. That $({\m B},\sigma)$ is an \sfmv-algebra follows promptly from item 1 and the fact that $\sigma$ is an interior operator. That $({\m B},\sigma,\neg\delta\neg)$ is a \tmv-algebra follows from items 1, 2, and 3.
\end{proof}

\begin{lemma}\label{lem: Embedding lemma S4L and S4tL}
Suppose that $({\m A},\Box)$ is an \sfmv-algebra and $({\m B},G,H)$ is a \tmv-algebra. Then both $\Box$ and $G$ are conuclei, and each of ${\m A}_\Box$ and ${\m B}_G$ is a GBL-algebra.
\end{lemma}
\begin{proof}
Each of $\Box$ and $G$ is a conucleus by definition. For each claim, it suffices to show that if $\m M$ is an MV-algebra and $\gamma$ is a conucleus on ${\m M}$ preserving $\cdot$ and $\meet$, then ${\m M}_\gamma$ is a GBL-algebra. For this, it is enough to show that ${\m M}_\gamma$ satisfies the divisibility identity $x\cdot (x\to y)\eq x\meet y$. Let $x,y\in M_\gamma$. Since ${\m M}$ is an MV-algebra, we have that $x\cdot^{\m M} (x\to^{\m M} y) = x\meet^{\m M} y$. Using the fact that $\gamma$ preserves $\cdot$ and $\meet$, and that $x,y$ are $\gamma$-fixed, we have:
\begin{align*}
x\cdot^{{\m M}_\gamma} (x\to^{{\m M}_\gamma} y) &= x\cdot^{\m M} \gamma(x\to^{\m M} y)\\
&= \gamma(x)\cdot^{\m M} \gamma(x\to^{\m M} y)\\
&= \gamma(x\cdot^{\m M} (x\to^{\m M} y))\\
&= \gamma(x\meet^{\m M} y)\\
&= \gamma(x)\meet^{{\m M}_\gamma} \gamma(y)\\
&= x\meet^{{\m M}_\gamma} y.
\end{align*}
This proves the claim.
\end{proof}

\subsection{The translations}\label{sec:main theorems}

We define a pair of translations $M$ and $T$ from the language $\mathcal{L}$ into the languages $\sm$ and $\tm$, respectively. We set $M(p)=\Box p$ for each $p\in\var$, $M(0)=0$, $M(1)=1$, and we extend $M$ recursively by
$$M(\varphi\star\psi) = M(\varphi)\star M(\psi) \text{, for }\star\in\{\meet,\join,\cdot\}\text{, and}$$
$$M(\varphi\to\psi) = \Box (M(\varphi)\to M(\psi)).$$
Further, if $\Gamma$ is a set of formulas of $\mathcal{L}$ then we define
$$M(\Gamma) = \{M(\varphi) : \varphi\in\Gamma\}.$$
The translation $T$ differs from $M$ only by replacing $\Box$ by $G$ and considering its codomain to be formulas of $\tm$ rather those those of $\sm$.

\begin{lemma}\label{lem:tMV and GBL}
Let $({\m A},\Box)$ be an \sfmv-algebra, and let $({\m B},G,H)$ be a \tmv-algebra.
\begin{enumerate}
\item Suppose that $h\colon \var\to ({\m A},\Box)$ is an assignment, and define $\bar{h}\colon \var \to {\m A}_\Box$ by $\bar{h}(p)=\Box(h(p))$. If $\varphi\in Fm_{\mathcal{L}}$, then $\bar{h}(\varphi)=h(M(\varphi))$.
\item If $\varphi\in Fm_{\mathcal{L}}$, then $\varphi\eq 1$ is valid ${\m A}_\Box$ if and only if $M(\varphi)\eq 1$ is valid in ${\m A}$.
\item Suppose that $h\colon \var\to ({\m B},G,H)$ is an assignment, and define $\bar{h}\colon \var \to {\m B}_G$ by $\bar{h}(p)=G(h(p))$. If $\varphi\in Fm_{\mathcal{L}}$, then $\bar{h}(\varphi)=h(T(\varphi))$.
\item If $\varphi\in Fm_{\mathcal{L}}$, then $\varphi\eq 1$ is valid ${\m B}_G$ if and only if $T(\varphi)\eq 1$ is valid in ${\m A}$.
\end{enumerate}
\end{lemma}

\begin{proof}
We will prove items 1 and 2. Item 3 follows by a proof identical to that of item 1 by replacing $\Box$ by $G$, $M$ by $T$, and $({\m A},\Box)$ by $({\m B},G,H)$. Similarly, item 4 follows from the same proof given for item 2. 

1. We argue by induction on the height of $\varphi$. If $\varphi$ is a constant or $\varphi\in\var$, then the statement is true by assumption. Now suppose that for all formulas $\varphi'$ of height strictly less than the height of $\varphi$ we have that $\bar{h}(\varphi')=h(M(\varphi'))$. If $\varphi=\psi\star\chi$ for $\star\in\{\cdot,\meet,\join\}$, then by definition $h(M(\varphi))=h(M(\psi\star\chi))=h(M(\psi)\star M(\chi))=h(M(\psi))\star h(M(\chi))$. By the inductive hypotheses, the latter is precisely $\bar{h}(\psi)\star \bar{h}(\chi)=\bar{h}(\psi\star\chi)=\bar{h}(\varphi)$ as desired. On the other hand, if $\varphi=\psi\to\chi$ then we have that $h(M(\varphi))=h(M(\psi\to\chi))=h(\Box(M(\psi)\to M(\chi)))=\Box(h(M(\psi))\to h(M(\chi)))$. By the inductive hypothesis, this term is equal to $\Box(\bar{h}(\psi)\to\bar{h}(\chi))=\bar{h}(\psi)\to^{{\m A}_\Box}\bar{h}(\chi) = \bar{h}(\psi\to\chi)=\bar{h}(\varphi)$. The result follows by induction.

2. Suppose first that $\varphi\eq 1$ is valid in ${\m A}_\Box$, and let $h\colon\var\to ({\m A},\Box)$ be an assignment. By item 1, $\bar{h}\colon\var\to {\m A}_\Box$ is an assignment in ${\m A}_\Box$ and $\bar{h}(\psi)=h(M(\psi))$ for all $\psi\in Fm_\mathcal{L}$. In particular, this shows that $h(M(\varphi))=\bar{h}(\varphi)=1$ since $\varphi\eq 1$ is valid in ${\m A}_\Box$, so as $h$ is arbitrary we have $M(\varphi)\eq 1$ is valid in ${\m A}$.

For the converse, suppose that $M(\varphi)\eq 1$ is valid in ${\m A}$ and let $h\colon\var\to {\m A}_\Box$ be an assignment. Because $A_\Box\subseteq A$, we may define a new assignment $k\colon\var\to ({\m A},\Box)$ by $k(p)=h(p)$ for all $p\in\var$. Since $M(\varphi)\eq 1$ is valid in ${\m A}$, we have $k(M(\varphi))=1$. By item 1, we have that $k(M(\varphi))=\hat{k}(\varphi)$, where $\hat{k}\colon\var\to {\m A}_\Box$ is defined by $\hat{k}(p)=\Box(k(p))$. Notice that since $k$ has its image among the $\Box$-fixed elements of $A$, we have for all $p\in\var$ that $\hat{k}(p)=\Box(k(p))=k(p)=h(p)$, and thus $\hat{k}=h$. From this we obtain that $h(\varphi)=\hat{k}(\varphi)=k(M(\varphi))=1$, so $\varphi\eq 1$ is valid in ${\m A}_\Box$.
\end{proof}

The following gives the main translation results of this paper.

\begin{theorem}\label{thm:translation}
Let $\Gamma\cup\{\varphi\}\subseteq Fm_{\mathcal{L}}$. Then:
\begin{enumerate}
\item $\{\psi\eq 1 : \psi\in\Gamma\}\vDash_{\sf GBL} \varphi\eq 1 \iff \{M(\psi)\eq 1 : \psi\in\Gamma\}\vDash_{\sf S4MV} M(\varphi)\eq 1.$
\item $\Gamma\vdash_{\bf GBL} \varphi \iff M(\Gamma)\vdash_{\m S4\emph{\luc}} M(\varphi).$
\item $\{\psi\eq 1 : \psi\in\Gamma\}\vDash_{\sf GBL} \varphi\eq 1 \iff \{T(\psi)\eq 1 : \psi\in\Gamma\}\vDash_{\sf S4_tMV} T(\varphi)\eq 1.$
\item $\Gamma\vdash_{\bf GBL} \varphi \iff T(\Gamma)\vdash_{\m S4_t\emph{\luc}} T(\varphi).$
\end{enumerate}
\end{theorem}

\begin{proof}
We first prove item 1. Suppose that $\{\psi\eq 1 : \psi\in\Gamma\}\vDash_{\bf GBL} \varphi\eq 1$, let $({\m A},\Box)$ be an \sfmv-algebra, and let $h\colon\var\to ({\m A},\Box)$ be an assignment. We aim to show $\{M(\psi)\eq 1 : \psi\in\Gamma\}\vDash_{\sf S4MV} M(\varphi)\eq 1$, so suppose that for all $\psi\in\Gamma$ we have $h(M(\psi))=1$. By Lemma~\ref{lem:tMV and GBL}(2) we have that $1=h(M(\psi))=\bar{h}(\psi)$. Since $\bar{h}$ is an assignment in ${\m A}_\Box$ (which is a GBL-algebra by Lemma~\ref{lem: Embedding lemma S4L and S4tL}), by hypothesis we have $\bar{h}(\varphi)=1$. Applying Lemma~\ref{lem:tMV and GBL}(2) again yields $h(M(\varphi))=1$, showing that $\{M(\psi)\eq 1 : \psi\in\Gamma\}\vDash_{\sf S4MV} M(\varphi)\eq 1.$

For the converse, suppose that $\{M(\psi)\eq 1 : \psi\in\Gamma\}\vDash_{\sf S4MV} M(\varphi)\eq 1$. Let ${\m A}$ be a GBL-algebra, let $h\colon\var\to {\m A}$ be an assignment, and suppose that $h(\psi)=1$ for all $\psi\in\Gamma$. It is enough to show that $h(\varphi)=1$. By Lemmas~\ref{lem:JipsenMontagna} and \ref{lem:poset products are tmv}, there exists an \sfmv-algebra $({\m B},\Box)$ such that ${\m A}$ embeds in ${\m B}_\Box$, and without loss of generality we may assume that this embedding is an inclusion. Using the fact that $A\subseteq B_\Box\subseteq B$, we define a new assignment $k\colon\var\to {\m B}$ by $k(p)=h(p)$ for all $p\in\var$. Notice that for all $p\in\var$ we have $\bar{k}(p)=\Box k(p) = \Box h(p) = h(p)$ since the image of $h$ consists of $\Box$-fixed elements, so by Lemma~\ref{lem:tMV and GBL}(2) we have $h(\chi)= k(M(\chi))$ for all $\chi$. In particular, $k(M(\psi))=1$ for all $\psi\in\Gamma$, and by the hypothesis we have $k(M(\varphi))=1$. But this implies $h(\varphi)=1$, proving the result.

Note that item 2 follows from Corollary~\ref{cor:alg application} since we have:
\begin{align*}
\Gamma\vdash_{\bf GBL}\varphi &\iff \{\psi\eq 1 : \psi\in\Gamma\}\vDash_{\sf GBL} \varphi\eq 1\\
&\iff \{M(\psi)\eq 1 : \psi\in\Gamma\}\vDash_{\sf S4MV} M(\varphi)\eq 1\\
&\iff M(\Gamma)\vdash_{\bf S4{\luc}} M(\varphi).
\end{align*}
Items 3 and 4 follows by proofs analogous to those given for items 1 and 2, respectively.
\end{proof}

As a final remark, we note that the temporal translation articulated in Theorem~\ref{thm:translation}(3,4) generalizes the translation offered in \cite{AGM2008}. \emph{G\"odel-Dummett logic} is the extension of propositional intuitionistic logic by the axiom scheme $(\varphi\to\psi)\join (\psi\to\varphi)$, and is algebraized by the variety of G\"odel algebras (which coincide with BL-algebras satisfying $x^2\eq x$). In \cite{AGM2008}, the authors deploy the temporal flow semantics (see \cite{ABM2009}) based on so-called bit sequences to exhibit a translation of G\"odel-Dummett logic into an axiomatic extension of Prior's classical tense logic. The present study was inspired by \cite{F2021}, which offers a relational semantics based on poset products that, among other things, generalizes the temporal flow semantics (see \cite[Section 4.2]{F2021}). Our development of the translations above can hence be thought of as extending the work of \cite{AGM2008} along the generalization offered in \cite{F2021}. Poset products give a powerful, unifying framework for inquiries of this kind, and we anticipate that they will find far-reaching application to translations. A thorough investigation of modal translations and modal companions for {\bf GBL} remains to be conducted, but we expect that the work in this paper to be an important preliminary step.

%
%
%

\bibliographystyle{splncs04}

\end{document}